\documentclass[12pt]{article}

\usepackage{amsfonts,epsfig,graphicx}
\usepackage{amsmath,amssymb,amsthm}
\usepackage{epsf}
\usepackage{graphics}
\usepackage{amsfonts}
\usepackage{amsmath}
\usepackage{amssymb}
\usepackage[usenames,dvipsnames]{color}
\usepackage{graphicx}
\usepackage{times}
\newtheorem{thm}{Theorem}
\newtheorem{prop}[thm]{Proposition}
\newtheorem{lem}[thm]{Lemma}

\newcommand{\uzero}{\underline{0}}

\def\EE{\mathbb{E}}

\def\PP{\mathbb{P}}

\def\R{\mathbb{R}}

\begin{document}

\title{
The Boolean Model in the Shannon Regime:\\
Three Thresholds and Related Asymptotics }

\author{
Venkat Anantharam\\ 
{\em EECS Department}\\
{\em University of California, Berkeley}\\
ananth@eecs.berkeley.edu\\
 \and
Fran\c{c}ois Baccelli\\
{\em Department of Mathematics and ECE Department}\\
{\em University of Texas, Austin}\\
and\\
{\em INRIA/ENS},
{\em Paris, France}\\
francois.baccelli@austin.utexas.edu {\em or} francois.baccelli@ens.fr}

\maketitle

\begin{abstract}
Consider a family of Boolean models, indexed by integers $n \ge 1$, where 
the $n$-th model features a Poisson point process in $\R^n$
of intensity $e^{n \rho_n}$ with $\rho_n \to \rho$ as $n \to \infty$,
and balls of independent and identically distributed
radii distributed like $\bar X_n \sqrt{n}$,
with $\bar X_n$ satisfying a large deviations principle.
It is shown that there exist three deterministic thresholds:
$\tau_d$ the degree threshold; $\tau_p$ the percolation threshold;
and $\tau_v$ the volume fraction threshold;
such that asymptotically as $n$ tends to infinity,
in a sense made precise in the paper:
(i) for $\rho < \tau_d$, almost every point is isolated, namely its ball
intersects no other ball; 
(ii) for $\tau_d< \rho< \tau_p$,
almost every ball intersects an infinite number of balls and
nevertheless there is no percolation;
(iii) for $\tau_p< \rho< \tau_v$,
the volume fraction is 0 and nevertheless percolation occurs;
(iv) for $\tau_d< \rho< \tau_v$,
almost every ball intersects an infinite number of balls and
nevertheless the volume fraction is 0; 
(v) for $\rho > \tau_v$, the whole space covered. 
The analysis of this asymptotic regime is motivated
by related problems in information theory, and may be of interest in other 
applications of stochastic geometry.
\end{abstract}

\section{Introduction}
The Boolean model was considered in high dimensions in a few papers,
both within the framework of stochastic geometry \cite{Gouere,PenroseCP}
and within the framework of information theory \cite{AB}.
The present paper discusses three thresholds and some
asymptotics related to these thresholds in a setting 
analogous to that in \cite{AB}, which is that where
the radii of the balls in the Boolean model scale with the dimension $n$
of the ambient space like ${\bar X_n} \sqrt{n}$,
where $({\bar X_n}, n \ge 1)$ is a sequence of random variables.
In this paper, we assume that this sequence of random variables
satisfies a large deviations principle (LDP). 

The first threshold is the volume fraction threshold, which
is the threshold at which the probability of coverage of the origin by the 
Boolean model switches from asymptotically vanishing to asymptotically approaching
$1$ as the dimension $n$ tends to $\infty$.
The second one is the percolation threshold; it was
first studied in detail in \cite{PenroseCP}
in the particular case where $\bar X_n$
is a constant. The case with random $\bar X_n$ was also discussed in
\cite{Gouere}. This is the threshold at which the probability of percolation in the 
Boolean model switches from asymptotically approaching $0$ to 
being asymptotically nonzero as $n \to \infty$.
The last is the degree threshold. This is the threshold at which the mean number of
grains of the Palm version of the Boolean model that intersect the grain of the 
origin switches from asymptotically being finite to asymptotically approaching 
$\infty$ as $n \to \infty$. It is not hard to see that these three thresholds 
are decreasing in the order in which they were presented.

The main new contributions of the present paper are (a)
representations of these three thresholds
in terms of optimization problems based on the rate function
of the LDP and (b) explicit asymptotics for various rates
of convergence in the neighborhood  of these thresholds.

\section{Setup}			\label{s.setup}

In each dimension $n \ge 1$ we have a homogeneous Poisson process of
intensity $e^{n \rho_n}$ (i.e. normalized logarithmic intensity $\rho_n$).
Assume that 
$\rho_n \to \rho$ as $n \to \infty$. Note that $\rho$ is a real number
(which can be negative). It is called the {\em asymptotic normalized
logarithmic intensity}
of this sequence of Poisson processes.
We will assume that the processes are defined on a single
probability space $(\Omega,{\cal F},\PP)$, although the coupling between the
different dimensions is not relevant for the issues we consider.
We will denote by $\PP^0_n$ the Palm probability of the Poisson point
process in dimension $n$.
See \cite[Chapter 13]{DVJ2} for the definition of and basic facts 
about Palm probabilities.

To each point $T^{(k)}_n$, $k \ge 1$, of the Poisson process in 
dimension $n$ (enumerated in some
way) we associate a mark $X^{(k)}_n\in \R^+$.
The $X^{(k)}_n$, $k \ge 1$, are assumed to be independent and 
identically distributed (i.i.d.) and
independent of the points.
For each dimension $n$, let $X^{(k)}_n \stackrel{d}{=} \bar{X}_n$ for all $k$.
Let $R^*_n$ denote $\EE[ \bar{X}_n]$, and let
$R^* := \lim_{n \to \infty} R^*_n$,  where
the last limit is assumed to exist. We assume that $0 < R^* < \infty$.
We assume 
that the sequence $(\bar{X}_n, n \ge 1)$ satisfies an LDP,
with good and convex rate function $I(\cdot)$ \cite{DZ}, 
e.g. by assuming the G\"artner-Ellis conditions. 
We also assume that the
following condition holds:
\begin{eqnarray}
\limsup_{n\to \infty} \EE [ (\bar X_n)^{\gamma n}]^{\frac 1 n} < \infty
\quad \mbox{for some $\gamma>1$}.\label{cond1}
\end{eqnarray}
By the deterministic setting we mean that
$\bar{X}_n$ is deterministic
and equal to $R^*_n$ for each $n \ge 1$, with $R^*_n \to R^*$ as 
$n \to \infty$. The deterministic setting is a special case of the general 
setting, but we will separately highlight the results in this case, since it is of
particular interest.

To the marked point process in dimension $n$, we associate
a Boolean model where the grain of point $T^{(k)}_n$ is a 
closed ball of radius $X^{(k)}_n \sqrt{n}$. Let
\begin{equation}
\label{eq:defbool}
 {\cal C}_n := \bigcup_k B(T^{(k)}_n, X^{(k)}_n \sqrt{n})
\end{equation}
denote this Boolean model, with $B(t,r)$ denoting the closed ball of center
$t\in \R^n$ and radius $r\ge 0$. Here, and in the rest of the paper, 
$:=$ denotes equality by definition.

From Slivnyak's theorem \cite[Chapter 13]{DVJ2},
the Palm version of the process in each dimension $n$ (i.e. its
law under $\PP^0_n$) is equivalent in law to 
the superposition of a stationary version of the process
and a process with a single point at the origin carrying a
ball with radius having law $\bar{X}_n \sqrt{n}$, and
independent of the stationary version (which is called the
reduced process of the Palm version). 

Our motivations for the analysis of this setting came from related problems
in information theory that we studied recently \cite{AB}. 
More specifically, in the study of 
error probabilities for coding over an additive white Gaussian 
noise channel \cite[Section 7.4]{Gallager}, it is natural to consider
a sequence of Poisson processes, one in each dimension $n \ge 1$, with 
well defined asymptotic logarithmic intensity, as was done in \cite{AB}, motivated by the 
ideas in \cite{P94}. The error exponent questions studied in \cite{AB} are 
related to 
\footnote{The scenario considered in \cite{AB} goes beyond additive white 
Gaussian noise to consider a setting where the additive
noise comes from sections of a stationary and ergodic process. The Boolean models that arise in the more general case involve grains, defined by the typicality sets of the 
noise process, that are not necessarily spherically symmetric.
Even more generally, in \cite{AB} the underlying point process in each dimension is 
allowed to be an arbitrary stationary ergodic process (while still requiring the 
existence of an asymptotic logarithmic intensity).}
the consideration of a Boolean model where the grains associated 
with the individual points are defined in terms of additive white Gaussian noise:
for all $n \ge 1$ and $k \ge 1$, 
let $W_n^{(i,k)}$, $n\ge i\ge 1$, denote an i.i.d. sequence of
Gaussian random variables, all centered and of variance $\sigma^2 > 0$.
Let $W_n^{(k)}$ denote the $n$-dimensional vector with coordinates
$W_n^{(i,k)}$, $n\ge i\ge 1$. Then $T_n^{(k)}+W_n^{(k)}$ 
belongs to the closed ball of center $T_n^{(k)}$ and radius 
$X_n^{(k)} \sqrt{n},$ with
$$ X_n^{(k)} := \left(\frac 1 n 
\sum_{i=1}^n \left(W_n^{(i,k)}\right)^2\right)^{\frac 1 2}$$
satisfying an LDP and all the assumptions listed above.
For the error exponent problem what is of interest 
is not this Boolean model, but the related Boolean model where 
the grain associated to each point is not the random ball described above, 
but rather an associated
{\em typicality region}, which in this case we may define as
the set
\[
\{ T_n^{(k)} + v ~:~ v \in \mathbb{R}^n,~
\| v \|_2 \le \sigma \sqrt{n} + \alpha_n \}~,
\]
where $\| v \|_2$ denotes the usual Euclidean length of $v$ and 
$0 < \alpha_n = O(\sqrt{n})$ are chosen such that 
\footnote{It is straightforward to check that it is possible to choose
$(\alpha_n, n \ge 1)$ satisfying these requirements.}
\begin{eqnarray*}
&& ~ \frac{\alpha_n}{\sqrt{n}} \to 0 \mbox{ as $n \to \infty$};\\ 
&&P( \| W_n^{(k)} \| \le 
\sigma \sqrt{n} + \alpha_n ) \to 1 \mbox{ as $n \to \infty$ 
(for each $1 \le k \le n$, of course)};\\ 
&& 
\frac{1}{n} \log \mbox{Vol} \{ v \in \mathbb{R}^n ~:~
\| v \|_2 \le \sigma \sqrt{n} + \alpha_n \} \to \frac{1}{2} \log (2 \pi e \sigma^2)
\mbox{ as $n \to \infty$}~.
\end{eqnarray*}
This fits within the class of deterministic Boolean models considered in this 
paper. Thus, having carried out the analysis in \cite{AB}, it was natural for us
to become
curious about the asymptotic properties in dimension of the
sequence of Boolean models with the grains being balls whose radii obey
a large deviations principle in the sense described above, 
and the current paper may be viewed as a start in 
that direction. In particular, it is to be hoped that this particular asymptotic regime,
which is so natural from an information theoretic viewpoint, will also be 
of value in the applications of stochastic geometry in other domains of science
and engineering.

The paper is structured as follows.
We start with a summary of results and heuristic explanations
in Section \ref{sec:res}.
We then give proofs in Section \ref{sec:proofs}.
For smoothness of exposition, we first discuss the volume fraction threshold, then the 
degree threshold, and finally the percolation threshold in each of these sections.
Some concluding remarks, making connections between the issues addressed here and
the information theoretic concerns of \cite{AB}, are made in Section 
\ref{s.concluding}, where in particular the instantiation of our general results in the
case of Gaussian grains is worked out in detail.


\section{Results}
\label{sec:res}

\subsection{Volume  Fraction  Threshold}

Consider the stationary version of the marked Poisson process
in each dimension. We are 
interested in the asymptotic behavior of the probability with which 
the origin is covered, namely $ \PP( \uzero \in \mathcal{C}_n)$.
 By stationarity, for any Borel set of $\R^n$, this probability
is also the mean fraction of the volume of the Borel set which is
covered by the Boolean model. 
We claim that there is a number $\tau_v$, 
called the {\em volume  fraction  threshold},
such that for $\rho < \tau_v$ this probability 
asymptotically approaches $0$ as $n$ tends to infinity, while for $\rho > \tau_v$ it 
asymptotically approaches $1$. The value of $\tau_v$ depends on the 
large deviations rate function $I(\cdot)$ associated to the sequence
of distributions of the radii of the marks.

The idea of the proof is based on the fact that
most of the volume of a ball is at the boundary. Hence  for
all $R>0$, 
the  mean  number of points at distance roughly
$R \sqrt{n}$ from the origin grows like 
\[
e^{n \rho_n} e^{ \frac{n}{2} \log ( 2 \pi e) + o(n)} e^{n \log R}~.
\]

Each such point covers the origin with probability $\PP(\bar{X}_n \ge R)$. 
For $R < R^*$ this probability is asymptotically $1$. For $R > R^*$ this probability 
decays like $e^{- n I(R) + o(n)}$, where $I(\cdot)$ denotes the rate function for 
the convergence $\bar{X}_n \stackrel{\PP}{\to} R^*$.

Let $\uzero$ denote the origin
in $\mathbb{R}^n$. We should therefore have
\[
\lim_{n \to \infty} \frac{1}{n} \log \PP( \uzero \in \mathcal{C}_n) = 
\rho + \frac{1}{2} \log ( 2 \pi e) + \sup_{R \ge R^*} (\log R - I(R))~,
\]
as long as 
\[
\rho + \frac{1}{2} \log ( 2 \pi e) + \sup_{R \ge R^*} (\log R - I(R)) < 0~,
\]
 where we used the fact that $I(R^*)=0$ which
implies that $\log(R)\le \log(R^*)-I(R^*)$ for $R\le R^*$.

Also
\[
\lim_{n \to \infty} \PP( \uzero \in \mathcal{C}_n) = 1~,
\]
if
\[
\rho + \frac{1}{2} \log ( 2 \pi e) + \sup_{R \ge R^*} (\log R - I(R)) > 0~.
\]

This gives a heuristic explanation of the value of
the threshold in the following theorem:

\begin{thm}\label{thm:vf}
Under the foregoing assumptions,
the volume fraction 
threshold is equal to
\begin{equation}			\label{Thresh.Vol}
\tau_v = - \frac{1}{2} \log ( 2 \pi e) + \inf_{R \ge R^*} (I(R) - \log R)~.
\end{equation}
More precisely, for $\rho < \tau_v$, 
as $n$ tends to infinity,
the volume fraction in dimension $n$, namely 
$\PP( \uzero \in \mathcal{C}_n)$, tends to 0 exponentially fast with
\begin{equation}
\label{eq:vfexp}
\lim_{n\to \infty} \frac 1 n
\log(\PP( \uzero \in \mathcal{C}_n)) =\rho-\tau_v,
\end{equation}
whereas for $\rho > \tau_v$, it tends to 1 with
\begin{equation}                \label{Limit.Above}
\lim_{n \to \infty} \frac{1}{n}
\log(- \log \PP( \uzero \notin \mathcal{C}_n)) 
= 
\rho - \tau_v~. 
\end{equation}
\end{thm}

Note that 
\begin{equation}
\label{eq:dettauv}
\tau_v \le - \frac{1}{2} \log ( 2 \pi e) - \log R^*.
\end{equation}
In the case of deterministic radii, i.e. 
when $\bar{X}_n$ equals the deterministic value $R^*_n$ for each
$n \ge 1$, with $R^*_n \to R^*$ as $n \to \infty$,
we have equality in eqn. (\ref{eq:dettauv}).
The R.H.S. of eqn. (\ref{eq:dettauv}) is identical 
to what is called the Poltyrev
threshold in \cite{AB}, where it in effect arose
in the context of the Boolean 
models with Gaussian grains truncated to their typicality regions,
as described at the end of Section \ref{s.setup}.
In Section \ref{s.concluding} we will discuss 
in more depth this
connection between the questions addressed in this paper and the information 
theoretic questions studied in \cite{AB}.
This threshold can also be described as
follows: the volume of the $n$-ball of random radius ${\bar X_n} \sqrt{n}$ 
scales like $e^{nV + o(n)}$ as $n$ tends to infinity, for some constant $V$; we have
$\tau_v=- V$ or equivalently the critical density
$e^{n\tau_v + o(n)}$ scales like the inverse of the volume of this $n$-ball.

\subsection{ Degree  Threshold}

We are interested in the number $D_n$ of 
points other than $\uzero$
whose ball intersects the ball of the point at the origin
under $\PP^0_n$. 


We claim that there is a number $\tau_d$, 
that we will call the {\em 
degree threshold},  such that 
if $\rho < \tau_d$, then  $\EE^0_n [D_n]$
asymptotically goes to $0$ when $n$ tends to infinity,
while for $\rho > \tau_d$ it asymptotically goes to $\infty$.


We argue as follows:
condition on the radius of the ball of
the point at the origin, call it $s \sqrt{n}$. Every point that lands in the 
ball of radius $s \sqrt{n}$ will have its ball meeting the ball of the origin. 
The number of such points grows like
\[
e^{n  \rho_n} e^{\frac{n}{2} \log (2 \pi e) + o(n)} e^{n \log s}~.
\]
Next consider points at a distance roughly $R \sqrt{n}$ from the (point at the) origin,
 with $R>s$.  The 
number of such points grows like 
\[
e^{n \rho_n} e^{\frac{n}{2} \log (2 \pi e) + o(n)} e^{n \log R}~.
\]
Each such point has its ball intersecting the ball of the point at the origin with 
probability asymptotically equal to $1$ if $R - s < R^*$ and with probability 
decaying like $e^{ - n I(R - s) + o(n)}$ if $R - s > R^*$. The 
number of points meeting the ball of the origin, conditioned on this ball having 
radius $s \sqrt{n}$, therefore grows like
\begin{equation}
\label{eq:meanasy}
e^{n (\rho + \frac{1}{2} \log (2 \pi e) + \sup_{R \ge s + R^*} (\log R - I(R - s))) + o(n)}~.
\end{equation}
The probability that the ball of the origin has radius roughly $s \sqrt{n}$ decays like
$e^{ - n I(s) + o(n)}~.$
Thus, the overall growth rate of the number of points whose
ball meets the ball of the origin is
\begin{eqnarray}	\label{Boolean}
&&~\sup_{s > 0} \left(  - I(s)  + 
\rho + \frac{1}{2} \log (2 \pi e) + \sup_{R \ge s + R^*} (\log R - I(R - s)) \right)
\nonumber \\
&&~~~ = \rho + \frac{1}{2} \log (2 \pi e) + \sup_{s > 0} \sup_{R \ge s + R^*}
\left(  - I(s)  + \log R - I(R - s)\right)
\nonumber \\
&&~~~ = \rho + \frac{1}{2} \log (2 \pi e) + \sup_{R > R^*} \left( \log R +
\sup_{0 < s \le R - R^*} \left(  - I(s) - I(R - s) \right) \right)
\nonumber \\
&&~~~ \stackrel{(a)}{=} \rho + \frac{1}{2} \log (2 \pi e) +
\max \left(
\sup_{R^* \le R < 2R^*} \left( \log R  - I(R - R^*) \right)\right.,
\nonumber \\ 
& & \hspace{7cm}
\left. \sup_{R \ge 2R^*} \left( \log R - 2 I(\frac{R}{2}) \right) \right)~,
\nonumber\\
&&~~~ = \rho + \frac{1}{2} \log (2 \pi e) + \sup_{R \ge 2 R^*} 
\left( \log R - 2 I(\frac{R}{2}) \right)~,
\end{eqnarray}
where in step (a) we have used the convexity of the rate function $I(\cdot)$ and the 
fact that $I(R^*) = 0$, and 
in the last step we have observed that the maximum in the first of the terms
in the overall maximum occurs at $R =2R^*$.
This gives intuition for the value of the threshold in the following theorem:

\begin{thm}
\label{thm:deg}
Under the conditions of Theorem \ref{thm:vf},
the degree threshold is 
\begin{equation}		\label{Thresh.Boolean.Old}
\tau_d = - \frac{1}{2} \log (2 \pi e) + \inf_{R \ge 2 R^*} 
\left( 2 I(\frac{R}{2}) - \log R \right)~.
\end{equation}
That is,
for $\rho < \tau_d$,
as $n$ tends to infinity, in 
dimension $n$, $\EE_n^0 [D_n]$ tends to 0 exponentially fast,
whereas for $\rho > \tau_d$ it tends to infinity exponentially fast.
In both cases,
\begin{equation}
\label{eq:degexp}
\lim_{n\to \infty} \frac 1 n
\log (\EE_n^0 [D_n]) =
\rho- \tau_d~.
\end{equation}
\end{thm}

It is sometimes more convenient to write the
 degree  threshold as
\begin{equation}			\label{Thresh.Boolean}
\tau_d = - \frac{1}{2} \log (2 \pi e) + \inf_{R \ge R^*} 
\left( 2 I(R) - \log (2 R) \right)~.
\end{equation}
Note that 
\begin{equation}
\tau_d \le - \frac{1}{2} \log (2 \pi e)   -\log (2 R^*)
\end{equation}
and that the R.H.S. of the last inequality is the
 degree  threshold in the case of deterministic radii \cite{PenroseCP}.


In the general case, the degree threshold can be described as
follows: let $\bar X_n'$ be an independent random variable
with the same law as $\bar X_n$. The volume of the
$n$-ball of random radius $(\bar X_n+\bar X_n') \sqrt{n}$
scales like $e^{nV + o(n)}$ as $n$ tends to infinity for some constant $V$; we have
$\tau_d=- V$ or equivalently the critical density
$e^{n\tau_d + o(n)}$ scales like the inverse of the volume of this $n$-ball.


\subsection{Percolation Threshold}


Consider the Palm version of the process in dimension $n$.
Consider the 
connected component of $\mathcal{C}_n$ that contains the origin, 
called the {\em cluster of the origin}, and denote the
set of points of the underlying Poisson process that
lie in this connected component by $K_n$.
The {\em percolation probability} in dimension $n$ is denoted by
\[
\theta_n := \PP^0_n( |K_n| = \infty)~,
\]
with $|A|$ the cardinality of set $A$.
This is one of the standard definitions for percolation probability 
in continuum percolation theory, see \cite[Section 1.4]{MeesterRoyBook}.

We are interested in the asymptotics of the percolation probability 
as $n \to \infty$. We claim that there is a number $\tau_p$, called the 
{\em percolation threshold}, such that for $\rho < \tau_p$ we have 
$\theta_n \to 0$ as $n \to \infty$, while for $\rho > \tau_p$ we have
$\liminf_n \theta_n > 0$. 


\begin{prop}
In the case of deterministic radii,
the percolation and the degree thresholds coincide,
i.e. $ \tau_p = \tau_d~$.
\end{prop}

To see that 
$\tau_p \ge \tau_d$, note that if $\rho < \tau_d$ then
$\EE^0_n[D_n] \to 0$ as $n \to \infty$ from Theorem \ref{thm:deg}.
It follows that $\PP^0_n(D_n = 0) \to 1$ as $n \to \infty$. Hence 
$\PP^0_n( |K_n| = 1) \to 1$ as $n \to \infty$,
from which it follows that 
$\theta_n \to 0$ as $n \to \infty$. This means $\rho < \tau_p$.
This argument actually works in the general case,
i.e. it does not require the 
assumption of deterministic radii.

To see that $\tau_p \le \tau_d$ in the case of deterministic radii,
we need to prove that if $\rho > \tau_d$
then $\liminf_n \theta_n > 0$. To this end, let us recall the main result of
\cite{PenroseCP}. In our notation, in \cite{PenroseCP} Penrose considers the 
sequence of Poisson Boolean models with deterministic radii 
$R^*_n = R^*$ for each $n \ge 1$, and
with normalized logarithmic intensities $\rho^y_n$ defined via 
\[
e^{n \rho^y_n} \frac{(\pi n)^{\frac{n}{2}}}{\Gamma(\frac{n}{2} + 1)}
(2 R^*)^n = y~,~~\mbox{ for all $n \ge 1$}~,
\]
where $y > 0$ is a fixed real number. 
Let $\theta^y_n$ denote the percolation probability in dimension $n$ with these
choices. The main result \cite[Theorem 1]{PenroseCP}
is that $\lim_{n \to \infty} \theta^y_n$ exists and 
equals the survival probability of a branching process with offspring distribution 
Poisson with mean $y$, and started with a single individual. In particular, this 
means that if $y > 1$, then $\liminf_{n \to \infty} \theta^y_n > 0$.

In our scenario with deterministic radii, the  degree  threshold 
(see eqn.  (\ref{Thresh.Boolean})) is 
\[
\tau_d = - \frac{1}{2} \log (2 \pi e) - \log (2 R^*)~,~~\mbox{ (deterministic radii)}~.
\]
It suffices to observe that if $\rho > \tau_d$, then
\[
\lim_{n \to \infty} e^{n \rho_n} \frac{(\pi n)^{\frac{n}{2}}}{\Gamma(\frac{n}{2} + 1)}
(2 R^*_n)^n = \infty~.
\]
That $\liminf_{n \to \infty} \theta_n > 0$ then follows from the 
result of \cite{PenroseCP} cited above. 

The main result on the case with random radii is: 

\begin{thm}
\label{thm:perc}
The percolation threshold is given by the formula
\begin{equation}		\label{Thresh.Perc}
\tau_p = - \frac{1}{2} \log (2 \pi e) + 
\inf_{R \ge R^*} \left( I(R) - \log(R + R^*) \right)~.
\end{equation}
That is, for $\rho<\tau_p$, $\theta_n\to 0$ when $n$ tends
to infinity, whereas for $\rho>\tau_p$ we have
$\liminf_n \theta_n >0$.
\end{thm}

Note that 
\[
\tau_p \le - \frac{1}{2} \log (2 \pi e) - \log (2 R^*)~.
\]
In the case of deterministic radii the 
minimum in the expression for the percolation 
threshold in eqn. (\ref{Thresh.Perc}) is achieved at $R = R^*$ and so we have
\[
\tau_p = - \frac{1}{2} \log (2 \pi e) - \log (2 R^*)~,~~\mbox{ (deterministic radii)}~.
\]
This also equals the value of the  degree  threshold in the case of deterministic 
radii.

The volume of the
$n$-ball of random radius $(\bar X_n+R^*) \sqrt{n}$
scales like $e^{nV + o(n)}$ as $n$ tends to infinity, for some constant $V$; we have
$\tau_p=- V$ or equivalently the critical density
$e^{n\tau_p + o(n)}$ scales like the inverse of the volume of this $n$-ball.
The intuition for this result is that what matters for percolation
is the mean number of balls that intersect a ball with typical
radius (namely roughly $R^*\sqrt{n}$):
if $\rho<\tau_p$, then on an event whose probability
tends to 1 as $n$ tends to infinity, namely the event that the ball of the point at $\uzero$ has
a radius in the interval $(R^*\sqrt{n}-\alpha_n, R^*\sqrt{n}+\alpha_n)$
for appropriate $0 < \alpha_n = O(\sqrt{n})$, no other ball intersects the latter ball asymptotically (because
$\rho<\tau_p$) and hence there is no percolation.
Conversely, for $\rho>\tau_p$, when the ball of $\uzero$ is at typicality, 
i.e. its radius lies in an interval of the kind defined above, 
we can consider a thinned version of the 
reduced process where we only retain points whose balls have radii that are 
at least above a threshold slightly less than the value of $R$ achieving the 
infimum in the definition of $\tau_p$ (assume for the moment that this 
infimum is achieved),
and we will still have that 
the 
mean number of balls intersecting the ball of the origin
tends to infinity like $e^{n\delta}$ with some
$\delta >0$. 
Since these balls themselves have radius at least as big as the typical 
ball of the unconditional distribution, this scenario propagates via a supercritical 
branching process, implying asymptotic percolation.
\footnote{For technical reasons, the formal proof looks slightly different from this
sketch, but this is the basic intuition.}

\subsection{Ordering of the Thresholds}		\label{ss.ordering}

\begin{thm}
\label{thm:order}
Under the foregoing assumptions,
\begin{equation}		\label{Relations.Betwn.Threshs}
\tau_d \le \tau_p \le \tau_v~.
\end{equation}
\end{thm}

\paragraph{Remark}
The ordering relation of the last theorem is not limited to the Poisson case.
The family of Boolean models considered here can naturally be extended to a
family of particle processes \cite{SW}, where the $n$-th particle process
features a stationary and ergodic point process $\mu_n$ in $\R^n$ with
normalized logarithmic intensity $\rho_n$ such that $\rho_n \to \rho$ as $n \to \infty$,
and i.i.d. marks satisfying the same independence and LDP assumptions as above.
This family of particle processes will be said to admit 
a volume fraction threshold $\tau_v$ if the associated ${\cal C}_n$,
still defined by eqn. (\ref{eq:defbool}), is such that
$\PP( \uzero \in \mathcal{C}_n)$
asymptotically approaches $0$ as $n$ tends to infinity
for $\rho < \tau_v$, while for $\rho > \tau_v$ it
asymptotically approaches $1$. 
Similarly, it will be said to admit a degree threshold $\tau_d$ if
the Palm expectation of $D_n$ tends to 0 as $n$ tends to infinity for $\rho < \tau_d$,
while for $\rho > \tau_d$ it
tends to $\infty$. 
The definition of the percolation threshold can also be extended verbatim.
Assuming that these three thresholds exist, then they must satisfy eqn. (\ref{Relations.Betwn.Threshs}).
This follows from first principles. If the volume fraction asymptotically
tends to 1, then percolation must hold asymptotically;
hence $\tau_p\le \tau_v$. If the mean number of balls that
intersect the ball of the origin tends to 0, then percolation
cannot hold asymptotically; hence $\tau_d\le \tau_p$.\\

Returning to the Poisson case, to better understand the thresholds, 
we first need to recall some facts from basic convex analysis
\cite{Rockafellar}.
Since it is a good convex rate function, $I(\cdot)$ is {\em proper}, as defined
in
\cite[pg. 24]{Rockafellar}.
Further, since it is lower semicontinuous, its epigraph is closed 
\cite[Thm. 7.1]{Rockafellar}, so $I(\cdot)$ is {\em closed} in the sense of
\cite[pg. 52]{Rockafellar}. Recall that the domain of $I(\cdot)$, defined as
the set of $R$ for which $I(R)$ is finite, is an interval, which is nonempty 
because $I(R^*) = 0$.
Since $I(\cdot)$ is closed and proper,
the right and left derivatives, $I^\prime_+(\cdot)$ and
$I^\prime_-(\cdot)$ respectively, are well-defined as functions on $\mathbb{R}$
(both defined to be $+\infty$ to the right of the domain of $I(\cdot)$ and 
to be $-\infty$ to the left of the domain of $I(\cdot)$). These are nondecreasing 
functions, each of which is finite on the interior of the domain of $I(\cdot)$, and
satisfy (\cite[Thm. 24.1]{Rockafellar}):
\[
I^\prime_+(z_1) \le I^\prime_-(x) \le I^\prime_+(x) \le I^\prime_-(z_2)~,~~
\mbox{ if $z_1 < x < z_2$}~,
\]
and, for all $x \in \mathbb{R}$,
\[
\lim_{z \uparrow x} I^\prime_-(z) = \lim_{z \uparrow x} I^\prime_+(z)
= I^\prime_-(x) \mbox{ and } 
\lim_{z \downarrow x} I^\prime_-(z) = \lim_{z \downarrow x} I^\prime_+(z)
= I^\prime_+(x)~.
\]
We further note that $0 \in [I^\prime_-(R^*), I^\prime_+(R^*)]$, since
$I(\cdot)$ is a nonnegative function with $I(R^*) = 0$.

This means we can define the following radii:
\begin{itemize}
\item $R_v \ge R^*$ as a value of $R$ satisfying 
\[
\frac{1}{R_v} \in [I^\prime_-(R_v), I^\prime_+(R_v)]~.
\]
Such $R_v$ achieves the infimum in 
eqn. (\ref{Thresh.Vol}) for the volume fraction threshold. 
Further, since $R \mapsto \frac{1}{R}$ is strictly 
decreasing and decreases to $0$ as $R \to \infty$, it follows that $R_v$ is
uniquely defined and finite.
\item $R_d \ge R^*$, as a value of $R$ satisfying 
\[
\frac{1}{2 R_d} \in [I^\prime_-(R_d), I^\prime_+(R_d)]~.
\]
Such $R_d$
achieves the minimum
in eqn.  (\ref{Thresh.Boolean}) for the 
degree  threshold. 
Further, since $R \mapsto \frac{1}{2 R}$ is strictly 
decreasing and decreases to $0$ as $R \to \infty$, it follows that $R_d$ is
uniquely defined and finite.
\item $R_p \ge R^*$, as a value of $R$ satisfying 
\[
\frac{1}{R_p + R^*} \in [I^\prime_-(R_p), I^\prime_+(R_p)]~.
\]
Such $R_p$
achieves the minimum in eqn.
(\ref{Thresh.Perc}) for the percolation threshold.
Further, since $R \mapsto \frac{1}{R + R^*}$ is strictly 
decreasing and decreases to $0$ as $R \to \infty$, it follows that $R_p$ is
uniquely defined and finite.
\end{itemize}
\begin{thm}
\label{thm:order2}
With the foregoing definitions, we have
\begin{eqnarray}			
\label{eq:total-order} R^*\le R_d \le R_p \le R_v \le R_p + R^* \le 2 R_d~.
\end{eqnarray}
\end{thm}

\section{Proofs}
\label{sec:proofs}
\subsection{Proof of Theorem \ref{thm:vf}}

Below we will use the {\em directed random geometric graph}
built as follows: its vertices are the nodes of the point process and
there is an edge from $T_n^{(k)}$ to $T_n^{(l)}$, $l\ne k$ if
$T_n^{(l)}\in B(T_n^{(k)}, X_n^{(k)} \sqrt{n})$.

Let $d^-_n$ denote the in-degree
of the node at the origin in this random directed graph under $\PP^0$,
namely the number of points whose ball 
contains the origin.
Let $d_n^+$ denote the out-degree
of the origin, namely the number of points which fall in the
ball of the point at the origin.
From the mass transport principle \cite{Last}, or by straightforward
elementary arguments based on an ergodic theorem for spatial averages, we have
$$\EE_n^0 [d^+_n] = \EE_n^0 [d^-_n].$$

Now, we have 
\begin{eqnarray*}
\EE_n^0[d_n^+] &\stackrel{(a)}{=}& 
e^{n \rho_n} \EE[ \frac{(\pi n)^{\frac{n}{2}}}{\Gamma(\frac{n}{2} + 1)}
\bar{X}_n^n]\\
&=& e^{n \rho_n} \EE[e^{\frac{n}{2} \log ( 2 \pi e) + o(n)} e^{n \log \bar{X}_n}]\\
&=& e^{n (\rho + \frac{1}{2} \log (2 \pi e)) + o(n)} \EE[ e^{n \log \bar{X}_n}]~.
\end{eqnarray*}

Here step (a) follows from
Slivnyak's theorem.
%
%
Then we have the following result.
\begin{lem}
\begin{equation}		\label{Asymp.Count}
\frac{1}{n} \log \EE_n^0[ d_n^-] \to 
\rho + \frac{1}{2} \log ( 2 \pi e) + \sup_{R \ge R^*} (\log R - I(R))~,~~\mbox{ as $n \to \infty$}~.
\end{equation}
\end{lem}
\begin{proof}
From what precedes,
\begin{eqnarray*}
\EE^0_n[d_n^-] =
e^{n (\rho +\frac{n}{2} \log ( 2 \pi e)) + o(n)}
\EE[e^{n \log \bar X_n}]~.
\end{eqnarray*}
It follows from Assumption (\ref{cond1})
and from Varadhan's lemma \cite{DZ} that
$$ \lim_{n\to \infty}  \frac 1 n 
\EE[e^{n \log \bar X_n}]
=  \sup_{R \ge R^*} (\log R - I(R)) ~,$$
where we have used the observation that $\log R - I(R) \le \log R^*$ for 
$0 \le R \le R^*$. This completes the proof.
\end{proof}

Now,  from the independent thinning theorem \cite[Exercise 11.3.1]{DVJ2},
the distribution of $d_n^-$ is Poisson.
Thus
\begin{equation}		\label{Prob.From.Count}
\PP( \uzero \in \mathcal{C}_n) = 1 - \exp( - \EE^0_n[d_n^-])\le
\EE^0_n[d_n^-]~.
\end{equation}
For $\rho <\tau_v$, we see from eqn.
 (\ref{Asymp.Count})  that 
\[
\EE[ d_n^-] \to 0 \mbox{ as $n \to \infty$}~,
\]
which implies that
\[
\PP( \uzero \in \mathcal{C}_n) \to 0 \mbox{ as $n \to \infty$}.
\]
In addition, for all $\alpha < 1$ we have 
\[
1 - \alpha x \ge \exp( -x) \ge 1 - x \mbox{ for all sufficiently small $x > 0$}~.
\]
Thus, from eqn. (\ref{Prob.From.Count}) we get that, for all $\alpha < 1$ and 
$n$ sufficiently large 
\[
\frac{1}{n} \log \EE^0_n[ d_n^-]  \ge
\frac{1}{n} \log \PP( \uzero \in \mathcal{C}_n) \ge \frac{1}{n} \log \EE^0_n[ d_n^-] 
+ \frac{1}{n} \log \alpha~.
\]
Thus we have
\begin{equation}		\label{Limit.Below}
\lim_{n \to \infty} \frac{1}{n} \log P( \uzero \in \mathcal{C}_n)
= \rho + \frac{1}{2} \log ( 2 \pi e) + \sup_{R \ge R^*} (\log R - I(R))~,
\end{equation}
which concludes the proof of eqn. (\ref{eq:vfexp}).

Suppose now that
$\rho > \tau_v$
Since 
\begin{equation}		\label{Prob.From.Count.2}
\PP( \uzero \notin \mathcal{C}_n) = \exp( - \EE^0_n[d_n^-])~,
\end{equation}
we then immediately get eqn. (\ref{Limit.Above}).

\subsection{Volume  Fraction 
Threshold for Deterministic Radii}

The proof above also works for the case of deterministic radii 
(equal to $R^*_n \sqrt{n}$ in dimension $n$ with $R^*_n \to R^*$ as
$n \to \infty$). The only change needed is to replace 
$\EE[ e^{n \log \bar{X}_n}]$ by  $e^{n \log R^*_n}$.
Then eqn.   (\ref{Asymp.Count})  is replaced by
\[
\frac{1}{n} \log \EE^0_n[ d_n^-] \to 
\rho + \frac{1}{2} \log ( 2 \pi e) + \log R^*~,
\]
and we learn that if
\[
\rho < - \frac{1}{2} \log ( 2 \pi e) - \log R^*
\]
then
\[
\lim_{n \to \infty} \frac{1}{n} \log \PP( \uzero \in \mathcal{C}_n)
= \rho + \frac{1}{2} \log ( 2 \pi e) + \log R^*~,
\]
which replaces eqn. (\ref{Limit.Below}), while if
\[
\rho > - \frac{1}{2} \log ( 2 \pi e) - \log R^*
\]
then
\[
\lim_{n \to \infty} \frac{1}{n} \log (-\log \PP( \uzero \notin \mathcal{C}_n)) 
= \rho + \frac{1}{2} \log ( 2 \pi e) + \log R^*~,
\]
which replaces eqn. (\ref{Limit.Above}).
\subsection{Proof of Theorem \ref{thm:deg}}		\label{ss.Boolean.Sketch.RR}

In each dimension $n$, consider the stationary version of the process.
Given $s > 0$, let $N_n(\uzero, s)$ denote the number of points whose
balls intersect the ball of radius $s \sqrt{n}$ centered at the origin
$\uzero$ in $\mathbb{R}^n$.  Then 
\begin{eqnarray*}
\EE[N_n(\uzero,s)] &\stackrel{(a)}{=}& 
e^{n \rho_n} \EE[ \frac{(\pi n)^{\frac{n}{2}}}{\Gamma(\frac{n}{2} + 1)}
(\bar{X}_n + s)^n]\\
&=& e^{n \rho_n} \EE[e^{\frac{n}{2} \log ( 2 \pi e) + o(n)} e^{n \log (\bar{X}_n + s)}]\\
&=& e^{n (\rho 
+ \frac{1}{2} \log (2 \pi e)) + o(n)} \EE[ e^{n \log (\bar{X}_n + s)}]~.
\end{eqnarray*}
 Step (a) again follows from Slivnyak's theorem
and the mass transport principle
applied to the directed graph 
with an edge from $T_n^{(k)}$ to $T_n^{(l)}$, $l\ne k$ if
$T_n^{(l)}\in B(T_n^{(k)},(X_n^{(k)}+s) \sqrt{n})$. 

Consider now the Palm version of the point process.
Recall that, from Slivnyak's theorem, ignoring the point at $\uzero$ 
and its mark leaves behind a stationary version of the 
marked point process, which is called the  reduced process. 
We therefore have,
\[
\EE^0_n[D_n] = \EE[ \EE[ N_n(\uzero, S) |S]]~,
\]
where $S \stackrel{d}{=} \bar{X}_n$ and is independent
of the reduced process.

Thus 
\[
\EE^0_n[D_n] = e^{n (\rho 
+ \frac{1}{2} \log (2 \pi e)) + o(n)} \EE[ e^{n \log (\bar{X}_n + \bar{X}^\prime_n)}]~,
\]
with $\bar{X}_n$ and $\bar{X}^\prime_n$ being i.i.d. 
We now use the fact that 
$\bar{X}_n + \bar{X}^\prime_n$
satisfies an LDP with good convex rate
function $2I(\frac u 2)$
to derive (\ref{Thresh.Boolean.Old}) from Varadhan's lemma.
For this, we have to check that if $\bar X'_n$ is a
variable with the same law as $\bar X_n$ and such that
$\bar X'_n$ and $\bar X_n$ are independent, then
\begin{eqnarray}
\limsup_{n\to \infty} \EE [ (\bar X_n+ \bar X'_n)^{\gamma n}]^{\frac 1 n} < \infty
\quad \mbox{for some $\gamma>1$}.
\label{cond2}
\end{eqnarray}
But this can be obtained from the following convexity argument
\begin{eqnarray*}
\EE [ (\bar X_n+ \bar X'_n)^{\gamma n}] & \le &
2^{\gamma n} \EE [(\bar X_n)^{\gamma n}]~,
\end{eqnarray*}
which implies that
\begin{eqnarray*}
\EE [ (\bar X_n+ \bar X'_n)^{\gamma n}]^{\frac 1 n} & \le &
2^{\gamma } \EE [(\bar X_n)^{\gamma n}]^{\frac 1 n}.
\end{eqnarray*}
Hence eqn. (\ref{cond2}) follows from eqn. (\ref{cond1}).


This completes the proof of the results on the 
 degree  threshold.

\subsection{ Degree  Threshold for Deterministic Radii}

The proof given above also works for the case of deterministic radii. 
The changes needed are analogous to those that were 
needed in the case of the volume  fraction  threshold.
\subsection{Proof of Theorem \ref{thm:perc}}

We need to prove two things: (1) if $\rho < \tau_p$ then 
$\lim_{n \to \infty} \theta_n = 0$ and (2) if $\rho > \tau_p$, then 
$\liminf_{n \to \infty} \theta_n > 0$.

To prove (1) we follow the lines of the proof in Section
\ref{ss.Boolean.Sketch.RR}.
Rather than considering
$\EE^0_n[D_n]$, we consider
$\PP^0_n(D_n > 0)$. As in Section \ref{ss.Boolean.Sketch.RR},
we write
\[
\PP^0_n(D_n > 0) = \EE^0_n [ \PP_n^0( D_n > 0 | S)]~,
\]
where $S \sqrt{n}$ now refers to the radius of the ball of the point at the 
origin. 
Now, if 
\[
\rho <  - \frac{1}{2} \log ( 2 \pi e) + \inf_{R \ge R^*} 
\left( I(R) - \log(R + R^*) \right)~,
\]
then, because $R_p$, as defined in Section \ref{ss.ordering} is finite, 
for sufficiently small $\epsilon > 0$ we also have
\[
\rho < - \frac{1}{2} \log ( 2 \pi e) + \inf_{R \ge R^*} 
\left( I(R) - \log(R + R^* + \epsilon) \right)~.
\]
On the event $\{ S \le R^* + \epsilon^\prime \}$, with
$\epsilon  > \epsilon^\prime > 0$, we have
\begin{eqnarray}
&&~ \lim_n \frac 1 n \log \left(\EE_n^0[D_n | S]\right)
\nonumber\\
&&~~~= \rho + \frac{1}{2} \log (2 \pi e) 
+ \sup_{R \ge R^* + S} \left( \log R - I( R - S) \right)\nonumber \\
&&~~~ = \rho + \frac{1}{2} \log (2 \pi e) 
+ \sup_{\tilde{R} \ge R^*} \left( \log( \tilde{R} + S) -  I(\tilde{R})
\right) \nonumber
\\
&&~~~ \le \inf_{R \ge R^*} 
\left( I(R) - \log(R + R^* + \epsilon) \right) + 
\sup_{\tilde{R} \ge R^*} \left( \log( \tilde{R} + S) - I(\tilde{R})
\right) \nonumber
\\
&&~~~ = \sup_{\tilde{R} \ge R^*} \left( \log( \tilde{R} + S) - I(\tilde{R}) \right)
- \sup_{\tilde{R} \ge R^*} \left( \log(\tilde{R} + R^* +\epsilon) - I(\tilde{R}) \right) \nonumber\\
&&~~~ < 0~.
\label{eq:argab}
\end{eqnarray}
We also have
\begin{eqnarray*}
\PP^0_n( D_n > 0) & = & \PP^0_n( D_n > 0 , S >  R^* + \epsilon^\prime  )\\
& & + \EE^0_n[ \PP^0_n(D_n > 0 | S) 1( S \le R^* + \epsilon^\prime)]\\
& \le & \PP (S >  R^* + \epsilon^\prime  )
+ \EE^0_n[ \EE^0_n[D_n  | S] 1( S \le R^* + \epsilon^\prime)]~.
\end{eqnarray*}
In the last expression, the first term has probability
asymptotically approaching $0$ as $n \to \infty$.
From eqn. (\ref{eq:argab}), for all $s$ in the integration interval,
the integrand in the second term tends pointwise to 0 as $n \to \infty$.
From this and dominated convergence, we conclude that
$\PP^0_n( D_n > 0 ) \to 0$ as $n \to \infty$, which proves (1).

We now prove (2),
assuming that 
\[
\rho > - \frac{1}{2} \log ( 2 \pi e) + \inf_{R \ge R^*} 
\left( I(R) - \log(R + R^* + \epsilon) \right)~.
\]
Let $R_p$, as defined in Section \ref{ss.ordering}, 
achieve the minimum in the definition of 
$\tau_p$ in eqn. (\ref{Thresh.Perc}). 
We need to distinguish between the two cases $R_p = R^*$ and 
$R_p > R^*$. 

Consider first the case $R_p = R^*$. Then 
\[
\rho = - \frac{1}{2} \log ( 2 \pi e) - \log(2 R^*) + \delta~,
\]
for some $\delta > 0$. This means that we can choose $\gamma > 0$ such that
\[
\rho > - \frac{1}{2} \log ( 2 \pi e) - \log(2 (R^* - \gamma)) + \frac{\delta}{2}~.
\]
For each dimension $n$ we consider the
thinned version of the reduced process, where we only retain the points whose
balls have radius at least $(R^* - \gamma) \sqrt{n}$.
We also consider only the event that the ball of
the point at the origin has radius at least $(R^* - \gamma) \sqrt{n}$.

Let $\tilde{\theta}_n$ denote the probability of percolation from the 
origin via its ball and through the balls of the 
thinned reduced point process, on the event that the ball of the origin has 
radius at least $(R^* - \gamma) \sqrt{n}$.
Since $\theta_n \ge \tilde{\theta}_n$, if we can show that
$\liminf_{n \to \infty} \tilde{\theta}_n > 0$, then we will
be done.

Let us now show that 
\begin{itemize}
\item[(a)] the probability that the ball of $\uzero$ has radius at least
$(R^*-\gamma)\sqrt{n}$ tends to 1 as $n$ tends to infinity;
\item[(b)] the probability that the number $J_n$ of balls
of the thinned point process intersecting the ball of $\uzero$
(with radius at least $(R^*-\gamma)\sqrt{n}$)
is positive tends to 1 as $n$ tends to infinity;
\item[(c)] the Boolean model with deterministic radii 
$(R^* - \gamma) \sqrt{n}$ for the points of the thinned reduced process percolates,
i.e. the associated percolation probability has a positive liminf.
\end{itemize}

Property (a) is immediate.
For proving (b), we show that
$\EE[J_n]$ tends to infinity with $n$, which will
complete the proof since $J_n$ is Poisson. 
The probability that the
$\bar X_n$ exceeds $(R^* - \gamma)\sqrt{n}$
is asymptotically $1$.
Hence, by arguments
similar to those used earlier,
$$ \lim_n \frac 1 n \EE[J_n] 
= \rho +\frac 1 2 \log(2\pi e) +\ln(2 (R^*-\gamma)) > \frac{\delta}{2}~,
$$
which completes the proof of (b).
For proving (c), we use the results in \cite{PenroseCP}.
The percolation threshold of the Boolean models with deterministic
radii $(R^* - \gamma)\sqrt{n}$ is 
$- \frac{1}{2} \log ( 2 \pi e) - \log( 2(R^* - \gamma))$.
Since the normalized logarithmic intensity of the thinned reduced process is 
still asymptotically $\rho$, and since this exceeds 
$- \frac{1}{2} \log ( 2 \pi e) - \log( 2(R^* - \gamma))$,
the proof of (c) follows from \cite{PenroseCP}.

The proof of the desired result, in the case $R_p = R^*$, now follows immediately
from (a), (b), and (c).

We next turn to the case $R_p > R^*$. Note that in this case we must have
$I(R_p) < \infty$. Since 
\[
\rho = - \frac{1}{2} \log ( 2 \pi e) + I(R_p) - \log(R_p + R^*) + \delta~,
\]
for some $\delta > 0$, we can choose $\gamma > 0$ such that 
$R^* < R_p - \gamma < R_p$ (which implies that $I(R_p - \gamma) < \infty$),
and such that 
\[
\rho > - \frac{1}{2} \log ( 2 \pi e) + I(R_p - \gamma) - \log(R_p + R^* - 2 \gamma) 
+ \frac{\delta}{2}~.
\]
For each dimension $n$ we consider the
thinned version of the reduced process, where we only retain the points whose
balls have radius at least $(R_p - \gamma) \sqrt{n}$.
We also consider only the event that the ball of
the point at the origin has radius at least $(R^* - \gamma) \sqrt{n}$.

Let $\tilde{\theta}_n$ denote the probability of percolation from the 
origin via its ball and through the balls of the 
thinned reduced point process, on the event that the ball of the origin has 
radius at least $(R^* - \gamma) \sqrt{n}$.
Since $\theta_n \ge \tilde{\theta}_n$, if we can show that
$\liminf_{n \to \infty} \tilde{\theta}_n > 0$, then we will
be done.

Let us now show that 
\begin{itemize}
\item[(a)] the probability that the ball of $\uzero$ has radius at least
$(R^*-\gamma)\sqrt{n}$ tends to 1 as $n$ tends to infinity;
\item[(b)] the probability that the number $J_n$ of balls
of the thinned point process intersecting the ball of $\uzero$
(with radius at least $(R^*-\gamma)\sqrt{n}$)
is positive tends to 1 as $n$ tends to infinity;
\item[(c)] the Boolean model with deterministic radii 
$(R_p - \gamma) \sqrt{n}$ for the points of the thinned reduced process percolates,
i.e. the associated percolation probability has a positive liminf.
\end{itemize}

Property (a) is immediate.
For proving (b), we show that
$\EE[J_n]$ tends to infinity with $n$, which will
complete the proof since $J_n$ is Poisson. 
The probability that the
$\bar X_n$ exceeds $(R_p - \epsilon)\sqrt{n}$
scales like $e^{-nI(R_p - \gamma) + o(n)}$.
Hence, by arguments
similar to those used earlier,
$$ \lim_n \frac 1 n \EE[J_n] 
= \rho +\frac 1 2 \log(2\pi e) - I(R_p - \gamma) + \ln(R_p + R^*- 2 \gamma)) > 
\frac{\delta}{2}~,
$$
which completes the proof of (b).
For proving (c), we use the results in \cite{PenroseCP}.
The percolation threshold of the Boolean models with deterministic
radii $(R_p - \gamma)\sqrt{n}$ is 
$- \frac{1}{2} \log ( 2 \pi e) - \log( 2(R_p - \gamma))$.
Since the normalized logarithmic intensity of the thinned reduced process is 
asymptotically $\rho - I(R_p - \gamma)$, and since this exceeds 
$- \frac{1}{2} \log ( 2 \pi e) - \log( R_p + R^* - 2 \gamma))$
(because $R_p > R^*$),
the proof of (c) follows from \cite{PenroseCP}.

The proof of the desired result, in the case $R_p > R^*$, now follows immediately
from (a), (b), and (c). This also completes the overall proof.


\subsection{Percolation 
Threshold for Deterministic Radii}

The proof of the percolation threshold in the case of deterministic radii (i.e. 
when the radii in dimension $n$ equal a constant $R^*_n \sqrt{n}$, with 
$R_n^* \to R^*$ as $n \to \infty$) can be completed in a much simpler way than 
the general proof. Since $\tau_p = \tau_d$ in the case of deterministic radii,
the absence of percolation when $\rho < \tau_p$ is an immediate consequence of 
the result proved earlier that $\PP(D_n > 0) \to 0$ as $n \to \infty$ when 
$\rho < \tau_d$. For the case $\rho > \tau_p$ the proof can be done in 
a way exactly as the general case where $R_p =R^*$ was handled above.

\subsection{Proof of Theorem \ref{thm:order}}
We first prove that $\tau_d\le \tau_p$.
By the convexity of the rate function $I(\cdot)$,
and because $I(R^*) = 0$, we have, for all $R\ge R^*$,
\[
2 I(\frac{R + R^*}{2}) \le  I(R)~.
\]
Hence, for all $R\ge R^*$,
\[
2 I(\frac{R + R^*}{2}) - \log(2 (\frac{R+R^*}{2})) 
\le  I(R) - \log(R + R^*)~.
\]
This establishes the result.

%
The fact that $\tau_p\le \tau_v$ immediately follows from
$$
\inf_{R \ge R^*} \left( I(R) - \log(R + R^*) \right) \le 
\inf_{R \ge R^*} \left( I(R) - \log(R) \right)~.
$$

\subsection{Proof of Theorem \ref{thm:order2}}
The fact that $R^*\le R_d$ is immediate from the definition of $R_d$.

We first prove that $R_d \le R_p$.
Recall that $R_d$ is uniquely defined by
\[
\frac{1}{2 R_d}\in
[I^\prime_-(R_d), I^\prime_+(R_d)]~,
\]
and $R_p$ is uniquely defined by 
\[
\frac{1}{R_p + R^*}\in 
[I^\prime_-(R_p), I^\prime_+(R_p)]~.
\]
Now 
\[
R_d > R_p \Longrightarrow 2 R_d > R_p + R^* \Longrightarrow 
\frac{1}{2R_d} < \frac{1}{R_p + R^*} \Longrightarrow R_d < R_p~.
\]
This contradiction implies that $R_d \le R_p$.

We next prove that $R_p \le R_v$. Here we also need to recall that
$R_v$ is uniquely defined by 
\[
\frac{1}{R_v} \in  [I^\prime_-(R_v), I^\prime_+(R_v)]~.
\]
Now
\[
R_p > R_v \Longrightarrow R_p + R^* > R_v \Longrightarrow 
\frac{1}{R_p + R^*} < \frac{1}{R_v} \Longrightarrow R_p < R_v~.
\]
This contradiction proves that $R_p \le R_v$.

We next prove that $R_v \le R_p + R^*$. For this, it suffices to 
observe the contradiction
\[
R_v > R_p + R^* \Longrightarrow \frac{1}{R_v} < \frac{1}{R_p+R^*} 
\Longrightarrow R_v < R_p \Longrightarrow R_v < R_p + R^*~.
\]

We finally prove that $R_p + R^* \le 2 R_d$. For this we observe the 
contradiction
\[
R_p + R^* > 2 R_d \Longrightarrow \frac{1}{R_p + R^*} < \frac{1}{2 R_d} 
\Longrightarrow R_p < R_d \Longrightarrow R_p + R^* < 2 R_d~.
\]

This completes the proof.


%
%
%

\section{Concluding Remarks}		\label{s.concluding}

\subsection{Connections with Error Exponents}
In this subsection we make some remarks about the connections between 
the concerns of this paper and the problem of error exponents in channel
coding over the additive white Gaussian noise channel
\cite[Section 7.4]{Gallager},
as discussed in \cite{AB} in the Poltyrev regime.

For all $n \ge 1$ and $k \ge 1$, 
let $W_n^{(i,k)}$, $n\ge i\ge 1$, denote an i.i.d. sequence of
Gaussian random variables, all centered and of variance $\sigma^2$.
Let $W_n^{(k)}$ denote the $n$-dimensional vector with coordinates
$W_n^{(i,k)}$, $n\ge i\ge 1$. Then $T_n^{(k)}+W_n^{(k)}$ 
belongs to the closed ball of center $T_n^{(k)}$ and radius 
$X_n^{(k)} \sqrt{n},$ with
\begin{equation}		\label{eq:GGrain}
X_n^{(k)} := \left(\frac 1 n 
\sum_{i=1}^n \left(W_n^{(i,k)}\right)^2\right)^{\frac 1 2} \stackrel{d}{=} \bar{X}_n~,
\end{equation}
where $\bar{X}_n$ denotes a random variable having the distribution of the 
normalized radius random variables in the preceding equation.
One can check that $(\bar{X}_n, n \ge 1)$ 
satisfy an LDP and all the assumptions listed above.
For each $\sigma^2 > 0$,
we call such a family of Boolean models (parametrized by $n \ge 1$, as usual)
the case with {\em Gaussian grains}.

For Shannon's channel coding problem in the Poltyrev regime, as considered in 
\cite{AB}, the focus is on the probability of error.
As a result, one only wants to associate those points in Euclidean
space that have a high probability of being of the type $T_n^{(k)} + W_n^{(k)}$
to the point $T_n^{(k)}$. 
Therefore one considers, instead of the Boolean model where 
a Gaussian grain is associated to each point, another Boolean model
where this Gaussian grain is replaced by an associated
{\em typicality region}, namely
the set
\[
\{ T_n^{(k)} + v ~:~ v \in \mathbb{R}^n,~
\| v \|_2 \le \sigma \sqrt{n} + \alpha_n \}~,
\]
where $\| v \|_2$ denotes the usual Euclidean length of $v$ and where
$0 < \alpha_n = O(\sqrt{n})$ are chosen such that 
\begin{eqnarray*}
&& ~ \frac{\alpha_n}{\sqrt{n}} \to 0 \mbox{ as $n \to \infty$};\\ 
&&P(\| W_n^{(k)} \| \le 
\sigma \sqrt{n} + \alpha_n ) \to 1 \mbox{ as $n \to \infty$ 
(for each $1 \le k \le n$)};\\ 
&& 
\frac{1}{n} \log \mbox{Vol} \{ v \in \mathbb{R}^n ~:~ 
\| v \|_2 \le \sigma \sqrt{n} + \alpha_n \} \to \frac{1}{2} \log (2 \pi e \sigma^2)
\mbox{ as $n \to \infty$}~.
\end{eqnarray*}
This now gives rise to a family of deterministic Boolean models
which will be referred to as the {\em truncated Gaussian grain} models below.
The {\em Poltyrev capacity} is the threshold for the asymptotic
logarithmic intensity of such a family of Boolean models. This threshold
is the asymptotic logarithmic intensity up to which it is possible
to make such an association
with asymptotically vanishing probability of error.
It is also the threshold up to which there is a vanishingly 
small probability that a point in Euclidean space
is covered by multiple truncated Gaussian grains, which
directly corresponds to what is called the volume 
fraction threshold in the present paper.

It turns out that 
the volume fraction threshold is smaller 
for Gaussian grains than
for truncated Gaussian grains, even though 
the normalized radii of the grains in the two models have the
same asymptotic limit $R^*$.
To see this, consider Gaussian grains with
per-coordinate variance $\sigma^2$,
as above. Then $\bar{X}_n^2$
is distributed as the average of $n$ i.i.d. squared Gaussian random variables 
of mean $0$ and variance $\sigma^2$, so we have 
$\bar{X}_n^2  \stackrel{\PP}{\to} \sigma^2$ as $n \to \infty$, which 
implies $\bar{X}_n \stackrel{\PP}{\to} \sigma$ as $n \to \infty$. This means
$E[ \bar{X}_n]$ (which is bounded above by $( E[ (\bar{X}_n)^2])^{\frac{1}{2}}
= \sigma$) converges to $\sigma$ as $n \to \infty$. This means
$R^* = \sigma$. The volume fraction threshold for deterministic grains with 
radius $R^* \sqrt{n}$ in dimension $n$ is then 
given by the R.H.S. of eqn. (\ref{eq:dettauv}).
The exponent of the 
growth rate in $n$ of the volume of each Gaussian grain is strictly bigger than 
this. That it is at least as big follows immediately from the convexity of 
the function $R \mapsto R^n$, defined for $R \ge 0$. To see the strict
inequality, first note
that the density of the radius of the Gaussian grain in dimension $n$ can be 
written as $g_n^\sigma(r)$, $r \ge 0$, where 
$g_n^\sigma(r) = \frac{1}{\sigma} g_n^1(\frac{r}{\sigma})$, with 
\[
g_1^\sigma(r) := \frac{n r^{n-1} e^{ - \frac{r^2}{2}}}
{2^{\frac{n}{2}} \Gamma(\frac{n}{2} + 1)}~,~~r \ge 0~,
\]
where $\Gamma(\cdot)$ denotes the standard Euler gamma function.
To figure out the asymptotic growth rate of the expected volume of a 
grain, we need to compute
\[
\lim_{n \to \infty} \frac{1}{n} \log \int_0^\infty V_n(1) r^n g_n^\sigma(r) dr~,
\]
where $V_n(1) := \frac{\pi^{\frac{n}{2}}}{\Gamma(\frac{n}{2} + 1)}$ denotes
the volume of the ball of unit radius in $\mathbb{R}^n$.

It is convenient to reparametrize $r$ as $v \sigma \sqrt{n}$, giving
\[
g_n^\sigma(v \sigma \sqrt{n}) = e^{- n( \frac{v^2}{2} - \frac{1}{2} - \log(v) + o(1))}~.
\]
Thus
\begin{eqnarray*}
&~& \lim_{n \to \infty} \frac{1}{n} \log \int_0^\infty 
\frac{\pi^{\frac{n}{2}}}{\Gamma(\frac{n}{2} + 1)} r^n g_n^\sigma(r) dr\\
~~~~~~~~~~~~~~~~ &=& \lim_{n \to \infty} \frac{1}{n} \log 
\left( (2 \pi e \sigma)^{\frac{n}{2}} \int_0^\infty e^{ n (\log v + o(1))}
e^{- n( \frac{v^2}{2} - \frac{1}{2} - \log(v) + o(1))} dv \right)\\
~~~~~~~~~~~~~~~~ &\stackrel{(a)}{=}& \frac{1}{2} \log ( 2 \pi e \sigma^2) + \frac{1}{2} ( \log 4 - 1)\\
~~~~~~~~~~~~~~~~ &>& \frac{1}{2} \log ( 2 \pi e \sigma^2)~,
\end{eqnarray*}
where step (a) comes from Laplace's principle that the asymptotics is 
controlled by the exponential term in the integrand with the largest 
exponent. Laplace's principle as just applied is only a heuristic, of course,
but this calculation makes the point that 
the volume fraction threshold for
Gaussian grains is strictly smaller than the volume fraction threshold for the 
truncated Gaussian grains (i.e. the Poltyrev threshold).

\subsection{Thresholds in the Gaussian Grain Case}
It is interesting to consider the case of Gaussian grains
in detail as an illustration of the general results in this paper,
and we turn to this next, giving,
in the process, a rigorous derivation of the volume fraction
threshold for this case as discussed in the preceding subsection.

We first need to determine the large deviations rate function
for the sequence $(\bar{X}_n, n \ge 1)$,
with each $\bar{X}_n$ defined as in eqn. (\ref{eq:GGrain}).
Here we think of $\sigma^2 > 0$ as being fixed.
It is easy to do this by first observing that $(\bar{X}^2_n, n \ge 1)$ 
satisfies the large deviations principle with
rate function $J(\cdot)$ given by
\[ 
J(z) = \begin{cases}
\frac{z}{2 \sigma^2} - \frac{1}{2} - \frac{1}{2} \log \frac{z}{\sigma^2} &
\mbox{ if $z > 0$}\\
\infty & \mbox{ otherwise}
\end{cases}~,
\]
which follows from the fact that if $Z_n$ is
a Gaussian random variable with mean zero and variance $\sigma^2$ then 
\[
\log E[ e^{\theta Z_n^2}] = 
\begin{cases}
 -\frac{1}{2} \log (1 - 2 \theta \sigma^2) & \mbox{ if $\theta < \frac{1}{2 \sigma^2}$} \\
\infty & \mbox{ otherwise}~.
\end{cases}
\]
The contraction principle \cite[Thm. 4.2.1]{DZ} then gives 
the rate function of the sequence $(\bar{X}_n, n \ge 1)$
as being $I(\cdot)$, where
\[
I(R) = \begin{cases}
\frac{R^2}{2 \sigma^2} - \frac{1}{2} - \frac{1}{2} \log \frac{R^2}{\sigma^2}
& \mbox{ if $R > 0$}\\
\infty & \mbox{ otherwise}~.
\end{cases}
\]
Another way to see this is to note that the convex conjugate
dual of this function is the function 
\[
\Lambda(\theta) = \frac{\theta \sigma}{2} 
\left( \frac{\theta \sigma + \sqrt{ \theta^2 \sigma^2 + 4}}{2} \right) 
+ \log \left( \frac{\theta \sigma + \sqrt{ \theta^2 \sigma^2 + 4}}{2} \right)~,
\]
and to check that 
\begin{eqnarray*}
\Lambda(\theta) &=& \lim_{n \to \infty} \frac{1}{n} \log E[ e^{n \theta \bar{X}_n}]\\
&=& \lim_{n \to \infty} \frac{1}{n} \log 
\int_0^\infty e^{\sqrt{n} \theta r} g_n^\sigma(r) dr\\
&=& \lim_{n \to \infty} \frac{1}{n} \log 
\int_0^\infty e^{ n \left( \theta v \sigma - \frac{v^2}{2} + \frac{1}{2} + \log v + o(1) 
\right)} dv~,
\end{eqnarray*}
by the use of Laplace's principle in the last expression on the right hand side.

The volume fraction threshold for the case of Gaussian grains (with
$\sigma^2 > 0$ being fixed) it then given by finding the solution
$R _v \ge \sigma$ to the equality
\[
\frac{R}{\sigma^2} - \frac{1}{R} = \frac{1}{R}~.
\]
There is a unique solution to this equation, namely 
$R_v = \sigma \sqrt{2}$, and this turns out to satisfy
$R_v \ge \sigma$, as it should. Here $\sigma$
is playing the role of $R^*$ in the general theory.
Substituting back into the formula
\[
\tau_v = - \frac{1}{2} \log( 2 \pi e) + I(R_v) - \log(R_v)
\]
for the volume fraction threshold gives
\[
\tau_v = - \frac{1}{2} \log(2 \pi e \sigma^2) - \frac{1}{2} ( \log 4 - 1)~.
\]
This is the announced rigorous derivation of the formula that was found 
above by the heuristic application of Laplace's principle.

The degree threshold is given by finding the solution $R _d \ge \sigma$
to the equality
\[
\frac{R}{\sigma^2} - \frac{1}{R} = \frac{1}{2 R}~.
\]
There is a unique solution to this equation, namely 
$R_d = \sigma \sqrt{\frac{3}{2}}$, and this turns out
to satisfy $R_d \ge \sigma$,
as it should. Substituting back
into the formula
\[
\tau_d = - \frac{1}{2} \log( 2 \pi e) + 2 I(R_d) - \log(2 R_d)
\]
for the degree threshold gives
\[
\tau_d = - \frac{1}{2} \log(2 \pi e \sigma^2) - \frac{1}{2} ( \log \frac{27}{2} - 1)~.
\]

The percolation threshold is given by finding the solution $R _p \ge \sigma$
to the equality
\[
\frac{R}{\sigma^2} - \frac{1}{R} = \frac{1}{R + \sigma}~.
\]
There is a unique solution to this equation, namely 
$R_p = \sigma c$, where $c$ is the unique root of the equation
\[
c^3 + c^2 - 2 c -1 = 0,
\]
which satisfies $c \ge 0$. In fact, this root satisfies $c > 1$.
\footnote{That there is a unique such root and that it satisfies $c > 1$
can be verified by noting that the expression on the left hand
side of this equation
equals $-1$ at $c =0$ and at $c =1$ and goes to $\infty$
as $c \to \infty$ and, further,
the derivative in $c$ of the expression  is
$3 c^2 + 2 c - 2$, which equals $-2$ at $c =0$ and is a convex function.}
Substituting back
into the formula
\[
\tau_p = - \frac{1}{2} \log( 2 \pi e) + I(R_p) - \log(R_p + \sigma)
\]
for the percolation threshold gives
\[
\tau_p = - \frac{1}{2} \log(2 \pi e \sigma^2) 
- \frac{1}{2} ( \log (c^2 (1 + c)^2) - c^2 + 1)~.
\]

Numerical evaluation of $c$ gives $1.2469796 < c < 1.2469797$.
This approximation suffices to verify that
\[
\log(\frac{27}{2}) - 1 > \log (c^2 (1 + c)^2) - c^2 + 1 > \log(4) - 1~,
\]
which confirms that 
$
\tau_d \le \tau_p \le \tau_v
$
in the case of Gaussian grains, as required by Theorem \ref{thm:order}.
This approximation also suffices to verify that
\[
1 < \sqrt{\frac{3}{2}} < c < \sqrt{2} < 1 + c < \sqrt{6}~,
\]
which confirms that
\[
R^*\le R_d \le R_p \le R_v \le R_p + R^* \le 2 R_d~,
\]
in the case of Gaussian grains,
as required by Theorem \ref{thm:order2}.

Returning to the Boolean model with truncated Gaussian grains
discussed in the last subsection,
since, for every $\sigma^2 > 0$, this family of Boolean models 
is a deterministic model
with $R^* = \sigma$, the rate function for this model
satisfies
\[
I(\sigma) = 0,\mbox{ and } I(R) = \infty \mbox{ for all $R \neq \sigma$}~.
\]
Thus, in this case the deterministic threshold equals the percolation threshold,
and both are $\log 2$ below the volume fraction threshold.
It is interesting to note, as observed in \cite{P94} and \cite {AB},
that this threshold also has a meaning; it is the threshold
for the asymptotic logarithmic intensity up to which the
truncated Gaussian grain of any given point of the Poisson 
process is so small that with probability asymptotically equal
to $1$ it does not meet the grain of any other point of the Poisson process.
This feature, which relates to a study of pairwise conflict in decoding between
two codewords, is central to Gallager's analysis of 
error exponents in the power constrained channels that are of interest to 
engineers; for more details see Section 7.4 of \cite{Gallager}
and in particular the study there of what is called Gallager's $E_0$ function.
\vspace{-.5cm}
\section*{Acknowledgements}
The research of the first author was supported by the 
ARO MURI grant W911NF- 08-1-0233, Tools for the
Analysis and Design of Complex Multi-Scale Networks, the 
NSF grants CNS-0910702 and ECCS-1343398, and the NSF Science
\& Technology Center grant CCF-0939370, Science of Information.
This work of the second author was supported by an award from
the Simons Foundation (\# 197982 to The University of Texas
at Austin).

\end{document}